\documentclass{amsart}

\usepackage{amsmath, amsthm, amsfonts, amssymb} 
\usepackage{hyperref}
\usepackage[T1]{fontenc}

\makeatletter 
\newtheorem*{rep@theorem}{\rep@title}
\newcommand{\newreptheorem}[2]{%
\newenvironment{rep#1}[1]{%
 \def\rep@title{#2 \ref{##1}}%
 \begin{rep@theorem}}%
 {\end{rep@theorem}}}
\makeatother

\newtheorem{theorem}{Theorem}[section]

\theoremstyle{definition}

\newtheorem{rmk}[theorem]{Remark}

\DeclareMathOperator{\rkc}{rank_C}

\DeclareMathOperator{\Vol}{Vol}
\DeclareMathOperator{\dist}{dist}
\DeclareMathOperator{\stab}{Stab}

\newcommand\Hy{\mathbb{H}}

\newcommand\hP{{\mathrm P}}
\newcommand\isom{\operatorname{Isom}}

\DeclareMathOperator{\GL}{GL}
\DeclareMathOperator{\Ort}{O}

\DeclareMathOperator{\PO}{PO}

\def\isom{\operatorname{Isom}(\mathbb{H}^n)}

\def\Hy{\mathbb{H}}
\def\Eucln1{\mathbb{E}^{\,n,1}}

\def\deg{\mathrm{deg}}
\def\rank{\mathrm{rank}}

\def\e_HW{\epsilon_{HW}}
\def\cO{{\mathfrak o}_k}

\newcommand\Q{\mathbb{Q}}
\newcommand\R{\mathbb{R}}

\begin{document}

\title[Effective bounds for Vinberg's algorithm]{Effective bounds for Vinberg's algorithm for arithmetic hyperbolic lattices}

\author{M. Belolipetsky}
\thanks{Belolipetsky is partially supported by CNPq and FAPERJ}
\address{IMPA, Estrada Dona Castorina, 110, 22460-320 Rio de Janeiro, Brazil}
\email{mbel@impa.br}
\author{M. Kapovich}
\address{Department of Mathematics, UC Davis, One Shields Avenue, Davis CA 95616, USA}
\email{kapovich@math.ucdavis.edu}


\begin{abstract}
A group of isometries of a hyperbolic $n$-space is called a reflection group if it is  generated by reflections in hyperbolic hyperplanes.
Vinberg gave a semi-algorithm for finding a maximal reflection sublattice in a given arithmetic subgroup of $\Ort(n,1)$ of the simplest type. 
We provide an effective termination condition for Vinberg's semi-algorithm with which it becomes an algorithm for finding maximal reflection sublattices. 
The main new ingredient of the proof is an upper bound for the number of faces of an arithmetic hyperbolic Coxeter polyhedron in terms of its volume. 
\end{abstract}

\maketitle

\section{Introduction}

The main goal of this paper is to supplement Vinberg's semi-algorithm with an effective termination bound. Together with this new bound it becomes an algorithm that determines if a given arithmetic hyperbolic lattice contains a maximal arithmetic reflection subgroup. At the same time, our main technical result (Theorem~\ref{thm1}) is non-algorithmic, it provides a {\em linear} upper bound on the number of Coxeter generators (denoted $\rkc(\Gamma)$) of a given arithmetic reflection lattice $\Gamma< \PO(n,1)$, in terms of the covolume of $\Gamma$ and the field of definition $k$ of the lattice $\Gamma$. More precisely, for each $n$, we obtain an explicit constant $C(n,\deg(k))$ such that if $\hP$ is a convex polyhedron in $\Hy^n$ of finite volume $V$, whose angles are submultiples of $\pi$, and the group $\Gamma$ generated by reflections in the facets of $\hP$ is arithmetic, defined over a number field $k$ of degree $\deg(k)$, then the number of facets of $\hP$ is $\le C(n, \deg(k)) \Vol(\hP)$.  

In the previous works the proof of non-reflectivity of some lattices was deduced from an existence of an infinite order symmetry of the fundamental polyhedron of the reflection subgroup (see e.g. \cite{Bug84}). Making Vinberg's algorithm to run sufficiently long allows one to detect enough vertices of the Coxeter diagram that are automorphism-equivalent and thus to prove that the polyhedron has infinite group of symmetries and is therefore of infinite volume. This method often works well in practice but there is no upper bound for the number of iterations required to detect the infinite symmetry. 

A similar result to our Theorem~\ref{thm1} was proved by Gelander \cite{Gelander},  who obtained a uniform linear upper bound $C_Gv$ on the minimal number of generators for lattices $\Gamma$ in a semisimple Lie group $G$ in terms of the covolume $v$ of $\Gamma$. However, his proof applied to a hyperbolic reflection group would not necessarily produce generators that are reflections. Note that the minimal number of generators of a Coxeter group $\Gamma$ (denoted $\rank(\Gamma)$) can be much smaller than $\rkc(\Gamma)$. The basic example is given by the symmetric groups $S_n$, for which we have $\rank(\Gamma) = 2 < n-1 = \rkc(\Gamma)$ if $n>3$. 

Note also that our bound depends on arithmetic invariants of reflective lattices. It can be made uniform (but nonlinear) as in Remark~\ref{rem1}, but even assuming the Lehmer conjecture would not suffice for our method to give a uniform linear upper bound. On the other hand, our Theorem~\ref{thm2} shows that we always have a uniform linear \emph{lower bound} for $\rkc(\Gamma)$ of a hyperbolic reflection lattice $\Gamma$ in terms of its covolume. There is no such bound for $\rank(\Gamma)$ as there are examples of lattices in $\Hy^n$ for $3\le n\le 8$ with a fixed number of generators and arbitrarily large covolume: In dimension 3 this is a consequence of existence of hyperbolic manifolds fibering over the circle, while in dimensions up to 8, this is a consequence of existence of hyperbolic manifolds of finite volume that {\em fiber algebraically}, see the recent paper~\cite{IMM}. 

\medskip
{\bf Acknowledgements.} This paper grew out of the discussion the authors had during the AIM workshop in December of 2020 on reflection groups. The authors are thankful to AIM for holding the workshop. We are also grateful to Alexander Kolpakov who corrected one of our miscalculations, and to Daniel Allcock and Nikolay Bogachev for numerous useful remarks. 

\section{Preliminaries}

In this section we briefly review the definitions and some properties of arithmetic groups of hyperbolic isometries and arithmetic hyperbolic reflection groups. We refer to a survey \cite{Bel16} for more details and references.

Let $\Eucln1$ be an $(n+1)$-dimensional vector space with the inner product defined by a quadratic form $f$ of signature $(n, 1)$. We have 
\[\{v \in \Eucln1 | (v, v) < 0\} = \mathfrak{C} \cup (-\mathfrak{C}),\]
where $\mathfrak{C}$ is an open convex cone. In the \emph{vector model}, the hyperbolic space $\Hy^n$ is identified with the set of rays through the origin in $\mathfrak{C}$, or $\mathfrak{C}/\mathbb{R}^+$, so that the group of isometries $\isom$ is given by the orthogonal transformations of $\Eucln1$ preserving $\mathfrak{C} $. 

Cofinite discrete subgroups of $\isom$ can be constructed using number theory. 
Let $k$ be a totally real number field with the ring of integers $\cO$. Consider a quadratic form $f$ of signature $(n,1)$ defined over $k$ such that for every non-identity embedding $\sigma : k\to\R$ the form $f^\sigma$ is positive definite. Let $\Gamma = \Ort_0(f,\cO)$ be the subgroup of the integral automorphisms of $f$ in  $H = \Ort_0(n,1)$, which is the full group of isometries of the hyperbolic $n$-space $\Hy^n$ (the group $\Ort_0(n,1)$ is the subgroup of the orthogonal group $\Ort(n,1)$ that preserves the cone $\mathfrak{C}$). By a classical theorem of Borel and Harish--Chandra, the group $\Gamma$ is a \emph{lattice} in $H$, i.e. it is discrete and has finite covolume.  Lattices obtained in this way and subgroups of $H$ which are commensurable with them are called \emph{arithmetic lattices of the simplest type}. The field $k$ is called the \emph{field of definition} of $\Gamma$ (and subgroups commensurable with it). The same terminology applies to the quotient orbifolds $\Hy^n/\Gamma$.

There are compact and finite volume non-compact arithmetic orbifolds. By the Godement's compactness criterion, an arithmetic group $\Gamma$ is non-cocompact if and only if $k = \Q$ and the associated quadratic form $f$ is isotropic. The Hasse--Minkowski theorem implies that for $k = \Q$ and $n \ge 4$ the latter condition holds automatically. Therefore, for $n \ge 4$ the quotient $\Hy^n/\Gamma$ is non-compact if and only if the arithmetic group $\Gamma$ of the simplest type is defined over the rationals. For $n = 2$ and $3$ the non-cocompact subgroups are defined over $\Q$ but there also exist cocompact arithmetic subgroups with the same field of definition.

We will mainly be interested in \emph{hyperbolic reflection groups}. These are the lattices in $\isom$ generated by finite sets of hyperbolic reflections in hyperplanes. Vinberg proved that arithmetic hyperbolic reflection groups are always of the simplest type~\cite{Vin67}. A convex polyhedron $\hP\subset \Hy^n$ which is a fundamental domain of 
a reflection group generated by reflections in facets of $\hP$, is called a \emph{Coxeter polyhedron}; such polyhedra 
are described by Coxeter diagrams. These diagrams are the graphs with vertices corresponding to the facets of $\hP$, and two different vertices $e_i$, $e_j$ are connected by a thin edge of integer weight $m_{ij} \ge 3$ or by $m_{ij}-2$ edges if the corresponding faces intersect with the dihedral angle $\frac{\pi}{m_{ij}}$, by a thick edge if they intersect at infinity (dihedral angle zero), and by a dashed edge if they are divergent. In particular, two vertices are not joined by an edge if and only if the corresponding faces of $\hP$ are orthogonal.

A discrete subgroup $\Gamma_0$ of a Lie group $H$ is called \emph{maximal} if it is not properly contained in any other discrete subgroup $\Gamma_1 < H$. It is well known that in a semisimple Lie group any lattice is contained in some maximal lattice. A hyperbolic reflection lattice $\Gamma< H=\isom$ is called a \emph{maximal hyperbolic reflection group} if $\Gamma$ is not properly contained in any other lattice in $H$ generated by reflections. Vinberg proved that the normalizer in $\isom$ of a maximal hyperbolic reflection group is a maximal lattice \cite{Vin67}. This fact is used in \cite{ABSW08} together with the other results to give an upper bound for the covolume of the maximal arithmetic reflection groups: 

\begin{theorem}\label{thm ABSW}
For each dimension  $n\ge 2$, there exists a computable constant $C(n)$ depending only on $n$ such that if  $\Gamma$ is 
a maximal arithmetic reflection subgroup of $\isom$, then $\Vol(\Hy^n/\Gamma) \le C(n)$.
\end{theorem}

The proof of this theorem follows from \cite[Section 6]{ABSW08}. In fact, a better bound $C_1 = C_1(n,k)$ can be obtained from the same argument as a function of the dimension $n$ and the degree of the field of definition $k$. It is then shown in \cite{ABSW08} that the degree of the field of definition is uniformly bounded in terms of $n$. Volumes of arithmetic hyperbolic orbifolds are discrete by the results Wang and Borel, hence Theorem~\ref{thm ABSW} implies finiteness of  maximal arithmetic reflection groups in dimension $n$.

\section{Bounds for the number of Coxeter generators}

Given a Coxeter group $\Gamma$,  we denote by $\rkc(\Gamma)$ the minimal number of its {\em Coxeter generators}, i.e. generators in a Coxeter system defining $\Gamma$ (such a Coxeter system need not be unique). In the case of interest, our group $\Gamma$ will be a reflection subgroup of $\isom$. The number of Coxeter generators, $\rkc(\Gamma)$,  is equal to the number of facets of a Coxeter polyhedron $\hP\subset \Hy^n$ of $\Gamma$. 

\begin{theorem}\label{thm1}
Let $\Gamma$ be an arithmetic hyperbolic reflection group defined over a field $k$. Then we have 
$$\rkc(\Gamma) \le C\Vol(\Hy^n/\Gamma),$$
where the constant $C = C(n,k)>0$ is computable and depends only on the dimension and the degree of the field of definition.
\end{theorem}
\begin{proof}
Let $u,v\in \Hy^n\cup \mathbb{S}^{n-1}$ be two distinct vertices of a Coxeter polyhedron $\hP$ of $\Gamma$. Consider the subgroup of $\Gamma$ generated by their stabilizers
$$ \Delta = \langle\stab_\Gamma(u), \stab_\Gamma(v)\rangle < \Gamma.$$ 
There are two possible cases:
\begin{itemize}
    \item[(1)] The group $\Delta$ is not virtually abelian. Then by the Margulis lemma we have 
    $$\dist(u,v) \ge \mu_n,$$ 
    where $\mu_n$ is the Margulis constant for $\Hy^n$.   
    \item[(2)] The group $\Delta$ is virtually abelian, i.e. is virtually infinite cyclic. 
       In this case every facet of $\hP$ passing though $u$ or $v$ is either orthogonal to $uv$ or contains $uv$. It follows that there is exactly one facet $F_u$ (resp. $F_v$) of $\hP$ passing though $u$ (resp. $v$) and orthogonal to $uv$. The group $\Delta$ then splits as a direct product of the infinite dihedral group generated by reflections in $F_u, F_v$ and the finite reflection group consisting of the elements of $\Gamma$ that pointwise fix $uv$. The product of reflections in $F_u$ and $F_v$ is a transvection (i.e. a hyperbolic element of $\Gamma$) along $uv$ with the displacement $2\dist(u,v)$. 
         
 There is a well known relation between the displacements of hyperbolic elements in arithmetic hyperbolic lattices and the Mahler measures of Salem numbers. We refer to a recent paper \cite{ERT19} for the details. In our case, the lattice $\Gamma$ is of the simplest type (by \cite{Vin67}), hence \cite[Theorem~5.2]{ERT19} together with Dobrowolski's lower bound for the Mahler measure of Salem numbers \cite{Dobr79} imply that 
    $$\dist(u,v) \ge c_k,$$ 
    where the constant $c_k > 0$ depends only on the degree of the field $k$. It is conjectured that in fact the constant does not depend on $k$. This is known as a special case of  {\em Lehmer's problem} and {\em Margulis' conjecture}. 
\end{itemize}
Therefore, in both cases we obtain 
    $$\dist(u,v) \ge d,$$
where $d=d(n,k)$ is a positive constant which depends only on the dimension $n$ and the degree of the field of definition $k$.

For a proper vertex $u$ of $\hP$ (i.e. a vertex which is in $\Hy^n$) 
we consider its neighborhood $B_u$ which is the intersection of $\hP$ with the open ball $B(u, d/2)\subset \Hy^n$. These neighborhoods are pairwise disjoint. The volume of $B_u$ is bounded below by $v_{n}(d/2)/f$, where $v_n(r)$ is the volume of a ball of radius $r$ in $\Hy^n$ and $f$ is an upper bound for the orders of finite subgroups of $\Gamma$. By \cite[Corollary~5.5]{ABSW08}, the orders of finite subgroups are bounded above by $f_{n,k} = m_n\deg(k)^{n(n+1)}$. Therefore, the number $V_1$ of the proper vertices of $\hP$ satisfies the upper bound
\begin{equation}\label{bound1}
V_1 \le \frac{\Vol(\Hy^n/\Gamma)}{v_{n}(d/2)/f_{n,k}}.
\end{equation}

For each ideal vertex $u$ of $\hP$ we consider a {\em canonical cusp neighborhood} $C_u\subset \hP$ introduced by Hersonsky in \cite{H93}. In our setting, it is defined as follows. For each horosphere $H\subset \Hy^n$ centered at $u$, its stabilizer $\stab_\Gamma(u)$ in $\Gamma$ acts on $H$ as a lattice; let 
$A_u< \stab_\Gamma(u)$ denote the (unique) maximal free abelian subgroup of rank $n-1$. We equip $H$ with the flat Riemannian metric induced from $\Hy^n$. There exists a unique horosphere $H_u$ as above such that the shortest nontrivial geodesic in $H_u/A_u$ 
has unit length. This horosphere $H_u$ bounds a horoball $B_u\subset \Hy^n$. The intersection $B_u\cap \hP$ is 
the canonical cusp neighborhood $C_u$ of $u$ in $\hP$. By \cite[Theorem~2.3]{H93}, the canonical cusp neighborhoods are pairwise disjoint. 

\begin{rmk}
While Hersonsky in \cite{H93} proves the disjointness only for discrete {\em orientation-preserving} groups of isometries 
of $\Hy^n$, passing to index 2 orientation-preserving subgroups yields the general result, since this passage does not alter the maximal free abelian subgroups $A_u$. 
\end{rmk}

By \cite[Proposition 3.4]{H93} (adapted to our setting of reflection groups), 
the volumes of $C_u$'s are bounded from below by \mbox{$\delta_{n-1}/((n-1)I_{n-1})$}, where $\delta_{n-1}$ and $I_{n-1}$ depend only on $n$. Here, $I_{n-1}$ is the least upper bound on the index of $A_u$ in $\stab_\Gamma(u)$, which equals the maximal order of a finite Coxeter group of rank $k=n-1$, of crystallographic type.

\begin{rmk}
For $k\notin \{2, 4, 6, 7, 8\}$, the maximal order is $2^k k!$, the order of the finite irreducible Coxeter group $B_k$. For the remaining ranks $k$, the maximal order is given by the order of the unique exceptional irreducible Coxeter group of rank $k$ of crystallographic type (i.e. $G_2, F_4, E_6, E_7$ and $E_8$ respectively). 
\end{rmk}

The constant $\delta_{n-1}$ is the largest lower bound on the volume of the flat torus $H_u/A_u$, where 
the volume is computed with respect to the flat metric on $H_u$ induced from $\Hy^n$. A crude lower bound for 
$\delta_{k}$ is the volume of the unit ball in the Euclidean space $\R^{k}$, i.e. 
$$
\frac{2\pi^{k/2}}{\Gamma(\frac{k}{2}+1)}. 
$$ 
Therefore, the number $V_2$ of the ideal vertices of $\hP$ satisfies
\begin{equation}\label{bound2}
V_2 \le \frac{\Vol(\Hy^n/\Gamma)}{\delta_{n-1} / ((n-1)I_{n-1})}.
\end{equation}

From \eqref{bound1} and \eqref{bound2} we conclude that the total number of vertices of $\hP$ satisfies
$$V \le \frac{\Vol(\Hy^n/\Gamma)}{c_{n,k}},$$
where the constant $c_{n,k}$ depends only on the dimension $n$ and the degree of the field of definition $k$.

Each finite vertex of an $n$-dimensional Coxeter polyhedron $\hP$ of finite volume is incident to  $n$ facets.  For ideal vertices of a Coxeter polyhedron the situation is a bit more complex. The link of this vertex in $\hP$ is the product of $k$ simplices of dimensions 
$n_i, i=1,...,k$, $n_1+...+n_k\le n-1$. The number of facets in such a product is at most 
$$
\sum_{i=1}^k (n_i+1)=  k+ \sum_{i=1}^k n_i\le 2(n-1),
$$
with maximum realized by links which are $(n-1)$-dimensional cubes. 
Since each facet of $\hP$ contains at least $n$ vertices, we have the following bound on the number of facets of $\hP$: 
\begin{equation}
F = \rkc(\Gamma) \le \frac{2(n-1)}{n} V.
\end{equation}
\end{proof}

\begin{rmk} \label{rem1}
More refined results about the relationship between the field of definition and covolume of an arithmetic lattice allow one to prove uniform bounds of the form 
\begin{equation}
\rkc(\Gamma) \le C\Vol(\Hy^n/\Gamma)^{1+\varepsilon}
\end{equation}
for any positive $\varepsilon$, where the constant $C=C_{n,\varepsilon}$ depends only on the dimension and~$\varepsilon$. To this end one can combine the preceding argument with the bound for displacement and volume of balls in terms of covolume (see \cite[Lemmas~3.1, 3.2]{Bel20}). The proofs of these bounds are nontrivial as they require some deep results from arithmetic groups theory and number theory. With the bounds at hand the remaining details are straightforward.   
\end{rmk}

In contrast with the minimal number of generators of a lattice, for Coxeter groups we have as well a lower bound for the rank in terms of volume. Although this bound is not required for the termination condition of Vinberg's algorithm, it may have an independent interest.

\begin{theorem}\label{thm2}
Let $\Gamma$ be a hyperbolic reflection lattice. Then we have
$$\rkc(\Gamma) \ge c \Vol(\Hy^n/\Gamma),$$
where the constant $c >0$ depends only on the dimension of the hyperbolic space. 
\end{theorem}
\begin{proof}
Let $\hP$ be the Coxeter polyhedron of $\Gamma$ with $F = \rkc(\Gamma)$ facets and $V$ vertices. By using the Klein model of the hyperbolic $n$-space we can realize $\hP$ as a convex polyhedron in $\R^n$ (this applies both to the compact polyhedra and to the ones with ideal vertices). Assume that $\hP \hookrightarrow\R^n$ is in general position so that all the vertices have different $n$-th coordinate and slice it by the hyperplanes $x_n = c$. Following the slicing down from the top vertex, by convexity and the bound for the number of faces of a Coxeter polyhedron that are incident to a vertex (see the proof of Theorem~\ref{thm1}), we deduce that the number $f(j)$ of the $j$-dimensional faces of $\hP$ satisfies
$$ f(j) \le c_{n,j} V \le c'_{n,j} F.$$
In particular, the total number of faces of $\hP$ is bounded linearly in terms of $F$. Hence the barycentric subdivision of $\hP$ consists of at most $c_nF$ simplices where the constant $c_n$ depends only on $n$. Convexity of a Coxeter polyhedron in $\Hy^n$ implies that the barycentric subdivision is a triangulation with geodesic simplices. The volume of a hyperbolic simplex is bounded above by a constant $\omega_n$ depending only on the dimension of the space. Therefore, we have 
$$\Vol(\Hy^n/\Gamma) = \Vol(\hP) \le \omega_n c_nF.$$ 
This gives the lower bound for the Coxeter rank with $\displaystyle c = \frac1{\omega_n c_n}$.
\end{proof}

\section{A termination condition for Vinberg's algorithm}

Let $f$ be an admissible quadratic form over a field $k$ with the ring of integers $\cO$, and let $\Gamma = \Ort_0(f,\cO)$ be the associated arithmetic subgroup. We would like to know when $\Gamma$ contains a maximal arithmetic hyperbolic reflection subgroup 
and, in particular, the form $f$ is \emph{maximal reflective}, 
as well as identifying such a subgroup.  The main practical tool for verifying reflectivity is \emph{Vinberg's algorithm} \cite[Section~3]{Vin72} which we shall briefly review now. We refer to \cite{Bel16} for the references regarding implementations and applications of the algorithm. More recent references are \cite{BP1, BP2, Bog} (a theoretical description of the algorithm, its Python implementation, and some applications),  \cite[\S 6.2.1]{G1}, \cite{G2}  (a general algorithm and its C++ implementation in special cases), 
\cite{Bot1, Bot2} (another general description of the algorithm and its Julia implementation for number fields whose rings of integers are PIDs, principal ideal domains), and \cite{A} (an alternative to Vinberg's algorithm). 

In the vector model of $\Hy^n$, a hyperplane is given by the set of rays in the cone $\mathfrak{C}$ which are orthogonal to a vector $e \in \Eucln1$ with $(e,e) > 0$. A hyperplane $\Pi_e$ defines two halfspaces, $\Pi_e^+$ and $\Pi_e^-$, where ``$\pm$'' is the sign of $(e,x)$ for $x$ in the corresponding halfspace, and a \emph{reflection}
\begin{equation*}
R_e: x \to x - 2\frac{(e,x)}{(e,e)}e,
\end{equation*}
where the Lorentzian inner product is given by $(u,v) = \frac{1}{2}(f(u+v)-f(u)-f(v))$.

We now describe briefly the algorithm, for details we refer to the references given above. 
For simplicity of the discussion, we shall assume that $\cO$ is a PID: this assumption ensures that the group $\Gamma$ is contained in the automorphism group of a free $\cO$-module of rank $n+1$, which simplifies the description. 
The vector $e$ corresponding to the reflection $R_e$ is defined up to scaling, so if $e$ has $k$-rational coordinates we can normalize it so that the coordinates are coprime integers in $\cO$. With this normalization we can assign to $R_e$ a parameter $s = (e,e)\in \cO$. The reflection $R_e$ belongs to the group $\Ort_0(f,\cO)$ if
$\frac{2}{s}(e,v_i) \in \cO, \text{ for the standard basis vectors } v_i$,  $i = 0,\ldots,n$.
This is called the \emph{crystallographic condition}. 

The algorithm starts with picking up a control vector $u_0\in {\mathfrak C}$ with integral coordinates which corresponds to a point $x_0\in {\Hy^n}$.  
Consider the reflection stabilizer subgroup $\Gamma_0$ in $\Gamma$ of the vector $u_0$. It is a (finite) group generated by $m \le n$ reflections in $\Gamma$ whose mirrors pass through $x_0$ ($m$ may be equal to $0$).

Let 
$$\hP_0 = \bigcap_{i=1}^m \Pi_{e_i}^-$$
be a fundamental chamber of $\Gamma_0$. All the halfspaces $\Pi_{e_i}^-$ are \emph{essential} (i.e. not containing the intersection of the other halfspaces). The corresponding vectors $e_i$ satisfy $(e_i, e_i) > 0$, $(e_i, u_0) = 0$ for all $i$, and the reflections $R_{e_i}$ generate $\Gamma_0$. There is a unique fundamental polyhedron $\hP$ of the maximal reflection subgroup of $\Gamma$ which sits inside $\hP_0$ and contains $x_0$.

The algorithm for finding $\hP$ proceeds inductively, by constructing a sequence of convex hyperbolic polyhedra $\hP_i$, $i>0$. The polyhedron $\hP_{i}$ is obtained by intersecting  $\hP_{i-1}$ with a halfspace $\Pi^-_{e_{m+i}}$. The vector  $e_{m+i}$ is chosen as follows: 
 
This is a normalized vector satisfying the crystallographic condition such that  
\begin{equation}\label{eq41}
(e_{m+i}, e_{m+i}) > 0, (e_{m+i}, u_0)<0, (e_{m+i}, e_j) \le 0
\end{equation}
for all $j < m+i$, i.e. 
$$\hP_{i} \subseteq \bigcap_{j=1}^{m+i} \Pi_{e_j}^-.$$
Moreover, we require that the distance between $x_0$ and $\Pi_{e_{m+i}}$ is the smallest possible, i.e. minimizing the value
\begin{equation}\label{eq42}
\sinh^2\left(\mathrm{dist}(x_0, \Pi_{e_{m+i}})\right) = -\frac{(e_{m+i},u_0)^2}{(e_{m+i},e_{m+i})(u_0,u_0)},
\end{equation}
subject to the rest of the conditions listed above. 
 
The minimality condition implies that all the hyperplanes $\Pi_{e_j}, j=1,...,m+i$ are essential, meaning that the intersection of any proper collection of half-spaces $\Pi^-_{e_j}$ is strictly larger than $\hP_{i}$. 

Note that if $k \neq \Q$, its integers do not form a discrete subset of $\R$. Nevertheless, in view of discreteness of the arithmetic group $\Gamma$ (as a subgroup of $\GL(n+1, \R)$), the distances $d(u_0, \Pi)$ (where $\Pi$ is the reflection hyperplane of a reflection in $\Gamma$) form a discrete subset $\mathcal{D} \subset \R$ (cf. \cite{Bug84}). 
Under some additional assumptions, such as $\cO$ being a PID, it is easy to find an explicit discrete subset of $\R$ containing $\mathcal{D}$. Now the algorithm runs through the discrete set of distances in an increasing order and for each of them finds all normalized vectors satisfying the crystallographic condition and inequalities \eqref{eq41}. Discreteness of $\Gamma$ implies that for each distance the number of such vectors is finite, moreover, formula \eqref{eq42} translates to a system of Diophantine equations that can be effectively solved (cf. \cite[Section 3.5]{BP1}). 

It may occur that the group $\Gamma$ does not contain any reflections. An example of such an arithmetic lattice defined over $k = \Q$ can be found in \cite[Section~4.2]{BK}. In this case the group $\Gamma_0$ is trivial and the set $\mathcal{D}$ is empty. Thus, the algorithm would not be able to start. In \cite{GS}, Grunewald and Segal gave a decision procedure to 
find integer solutions for quadrics $Q(x_1,....,x_{n+1})=0$ (with integer coefficients), subject to finitely many specified congruences and linear inequalities. It can be applied to give an algorithmic description of the set $\mathcal{D}$ which, in particular, tells, after a finite number of steps, if this set is empty or not. We expect that the method from \cite{GS} will generalize to solution of the systems of quadratic equations in algebraic integers but the details of this need to be carefully checked.

When the group $\Gamma$ contains a reflection it has infinitely many of them. The algorithm terminates if it yields a polyhedron $\hP = \bigcap_{i} \Pi_{e_i}^-$ of finite volume, in which case the form $f$ is reflective. Moreover, the subgroup $\Gamma_R$ generated by reflections in the faces of $\hP$ is a maximal reflection subgroup of $\Gamma$ by \cite[Propositions~4 and~5]{Vin72}.


The finiteness of volume of $\hP$ can be checked as follows. Each hyperbolic half-space $\Pi^-_{e_i}$ defining $\hP$ is the intersection of the linear half-space $\widetilde\Pi^-_{e_i}\subset {\mathbb E}^{n,1}$ with the upper sheet of the hyperboloid $\{v: (v, v)< 0\}\cap {\mathfrak C}$. Thus, one computes the extreme rays of the convex polyhedral cone 
$$
\bigcap_{i} \widetilde\Pi^-_{e_i} \subset {\mathbb E}^{n,1}. 
$$
Then $\hP$ has finite volume if and only if each extreme ray is contained in the closure of ${\mathfrak C}$, i.e. its generating vector $v_i$ satisfies the inequality $(v_i, v_i)\le 0$.

Now note that each iteration of the algorithm adds a new facet to $\hP_{i-1}$. Hence Theorems~\ref{thm ABSW} and \ref{thm1} give an upper bound on the number of iterations 
which depends only on the dimension $n$ and the degree of $k$. Moreover, in view of Remark~\ref{rem1}, there exists an upper bound which depends only on $n$. If more facets are produced, then $\Gamma = \Ort_0(f,\cO)$ does not contain a maximal arithmetic hyperbolic reflection group.

\end{document}